\documentclass{article}
\usepackage[english]{babel}
\usepackage{amsmath, amsthm, amsfonts, amssymb, csquotes}
\usepackage{amscd}
\usepackage{hyperref}
\usepackage{caption}
\usepackage{subcaption}
\usepackage{amsxtra}
\usepackage{floatflt}
\ifx\pdfoutput\undefined
\usepackage{graphicx}
\else
\usepackage[pdftex]{graphicx}
\fi

\hypersetup{colorlinks=true, linkcolor=red, citecolor=red}

\newtheorem{thm}{Theorem}

\newtheorem{lem}[thm]{Lemma}
\newtheorem{prop}[thm]{Proposition}

\newtheorem{remark}{Remark}
\theoremstyle{remark}
\theoremstyle{definition}

\newcommand{\abs}[1]{\left\vert#1\right\vert}

\title{Focus singularities}

\begin{document}
\begin{center}

\Large{{\bf Focus-focus singularities in classical mechanics}}

\vspace{5mm}

\normalsize

{\bf Gleb Smirnov}

\vspace{5mm}

{Lomonosov Moscow State University, Department of Differential Geometry and Its Application, Moscow}

\vspace{2mm}

{glebevgen@yandex.ru}
\end{center}

\vspace{3mm}

\begin{abstract}
In this paper the local singularities of integrable Hamiltonian systems with two degrees of freedom are studied. The topological obstruction to the existence of a focus-focus singularity with the given complexity is found. It is shown that only simple focus-focus singularities can appear in a typical mechanical system. Model examples of mechanical systems with complicated focus-focus singularity are given.
\end{abstract}

\section*{Introduction}\label{sec1}

An integrable Hamiltonian system with two degrees of freedom is a symplectic 4-manifold $(M^4,\omega)$ (the phase space) with two functions $H,F:M^4 \rightarrow \mathbb{R}$ (the first integrals), which are linearly independent almost everywhere, and they are in involution $\left\{H,f\right\} = 0$ with respect to the corresponding Poisson bracket. The vector field $\text{sgrad}\,H$ defined by the following condition
$$
\text{d}\,H(\boldsymbol{\eta}) = \omega^2 (\boldsymbol{\eta}, \text{sgrad}\,H)\ \text{for arbitraty vector field $\boldsymbol{\eta}$} 
$$
is called \emph{Hamiltonian vector field}. 

The level $\left\{H = h\right\}$ of the function $H$ is called an \emph{energy level}. We will assume that all energy levels are compact. Mapping $\mathcal{F}: M^4 \rightarrow \mathbb{R}^2(h,f)$, $\mathcal{F}(x) = \left(H(x), F(x) \right)$ is called a \emph{moment map}. The moment map defines a foliation on $M^4$ called a \emph{Liouville foliation}, whose fibres are connected components of its preimages $\mathcal{F}^{-1}(h, f)$.

A point $x\in M^4$ is called a \emph{singular point of rank} $i$ ($i=0,1$) of Hamiltonian system if $\text{rank}\,\text{d}\,\mathcal{F}(x) = i$ and the fiber containing it is called a \emph{singular fiber}.

If the rank of the moment map is maximal at all points of the fiber, then this fibre is called \emph{regular} and according to Liouville theorem it is diffeomorphic to 2-dimensional torus $\mathbb{T}^2$ (if it is compact). Thus, the phase space of an integrable system is a fibered manifold, whose fibres are invariant submanifold. Almost all of them are tori.

If we want to investigate the qualitative behavior of an integrable system, we must study the topology of the corresponding Lioville foliation. Since all fibres in the neighbourhood of a regular fiber have got the same structure (trivial $\mathbb{T}^n$ - bundle), the topology is mainly determined by the singularities of our system.

The description of the basic concepts of Hamiltonian mechanics is given, for example, in \cite{arnold}.
The topological methods are described in details in \cite{smeil, kharlamov_book, igs1, cushman_book} (see also \cite{shema, lerman, zung3, bolsinov, borisov, oshemkov}).

The local classification of singular points is given by Eliasson theorem (see \cite{eliasson, bolsinov}). In this work we will study one type of singular points. A singular point $x$ of rank 0 is called a \emph{focus-focus singularity}, if the Hessians $\text{d}^2\!H(x), \text{d}^2\!F(x)$ are linearly independent, and the roots of the polynom
$$
P(\lambda) = \text{det}\!\left( a\,\text{d}^2\!H(x) + b\,\text{d}^2\!F(x) - \lambda \Omega(x) \right)
$$
are distinct and have a form $\pm x \pm i y$, where $x, y \neq 0$ for a certain linear combination $a\,\text{d}^2\!H(x) + b\,\text{d}^2\!F(x)$. Here $\Omega$ is a matrix of the symplectic form.

Recall the basic properties of focus-focus singularities, see details in
\cite{matveev, zung1, zung2, svn, igs1, bolsinov, izosimov}.
Assume that the point $x$ is a singular point of a focus-focus type. Then \cite{vey,ito}
there exist symplectic coordinates $p_1, q_1, p_2, q_2$, $x = (0,0,0,0)$ in a certain neighborhood of $x$ such that in these coordinates the first integrals $H, F$ become the functions of a pair $f_1, f_2$ of \emph{canonical integrals}, namely,
$$
\displaystyle{ H = H(f_1, f_2),\ F = F(f_1, f_2),}
$$
where
$
\displaystyle{ f_1 = p_1 q_1 + p_2 q_2,\ f_2 = p_1 q_2 - p_2 q_1,}
$
and
$\displaystyle{ \omega = \text{d}\,p_1 \wedge \text{d}\,q_1 + \text{d}\,p_2 \wedge \text{d}\,q_2}$.

Let $L$ be a singular fiber of Liouville foliation, which contains one or several points of a focus-focus type.
If a singular fiber contains no points of rank 1, and all its points of rank 0 are points of a focus-focus type, then we will call $L$ a \emph{focus fiber}. In this case we will say that the integrable system has a focus-focus singularity.

\begin{floatingfigure}[lrp]{5cm}
\centering
\includegraphics{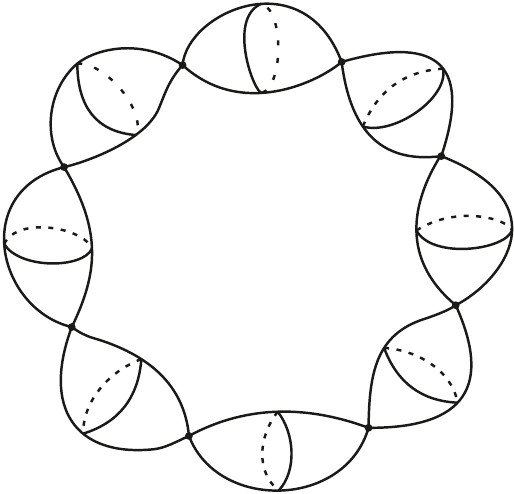}
\caption{Focus fiber}
\label{fig:figlabel}
\end{floatingfigure}

The number of singular points of the fiber will be called a \emph{complexity} of focus singularity. In case of complexity 1 we will say that the singularity is \emph{simple}. Denote these points as $x_1,\ldots,x_n$. The singular fiber is a sequence of embedded Lagrangian 2-spheres $L_i$,
each of them has a pair of selected points $x_i,x_{i+1}$. 
At every point $x_i$ two neighboring spheres $L_{i-1}$ и $L_i$ intersect transversally inside a 4-dimensional manifold. The singular fiber $L$ is a limit of a family of regular tori.
This limit is described as follows: we fix $n$ parallel non-trivial cycles on a regular torus, each of these cycles contracts to a point. As a result we obtain a pinched torus with $n$ constrictions (fig. 1).

The focus-focus singularities appear in many integrable mechanical systems. As an example, we will mention two systems of rigid body dynamics: Clebsch system \cite{pogosyan, morozov} and Lagrange system \cite{oshemkov}. Other examples are the spherical pendulum system \cite{cushman_book} and the problem of a motion of ellipsoid on a smooth surface \cite{ivo4kin, kazakov}.
Except for the latter system, all focus-focus singularities in these cases are simple. But at the same time we can construct a model example of a singularity of any complexity \cite{matveev} (see also \cite{igs1}). One can assume that there are some topological obstructions to the existence of a focus-focus singularity with large value of complexity on a given symplectic manifold. In the next section we will discuss the well-known question: what singularities can the given symplectic manifold contain? In such form this question was formulated in \cite{bolsinov}.

\section*{Topological obstruction to the existence of the focus-focus singularity of the given complexity}\label{sec2}

Recall the neccesary information from algebraic topology \cite{Bt, hatcher}. Let $S$ be a closed compact oriented 2-submanifold in an oriented 4-manifold $M$. Then this submanifold has a class of De Rham cohomology with compact support $\left[ \eta_S \right] \in H_{c}^{2}(M)$, which is defined uniquely by the condition
\begin{equation}\label{eq:form_2_1}
\int_S \omega = \int_M \omega \wedge \eta_S\quad \text{for any closed $2$-form $\omega$.}
\end{equation}
This class is called Poincare dual to a submanifold $S$.
Each cohomology class in $H^{2}_{c}(M)$ defines naturally a class in $H^{2}(M)$,
and broadly speaking, it can be trivial. But in this work these details are unimportant, and later we will consider class $\left[\eta_S\right]$ as a cohomology class in $H^{2}(M)$.
For any closed compact oriented 2-dimensional submanifold $X$, which intersects transversally with $S$, the value of the integral $\int_X \eta_S$ is denoted by by $X\cdot S$ and is called algebraic number of points of intersection of $X$ and $S$.

The algebraic number of points of intersection (\emph{index of intersection}) of two transversally intersected 
submanifolds $X$ and $S$ also can be defined in purely geometrical terms.
The intersection $S \cap X$ consists of a finite number of points $x_1,\ldots,x_n$. Denote as $\tau_{(i)}^{S}$ an orienting basis, which is tangent to $S$ at the point $x_{i}$, and as $\tau_{(i)}^{X}$ an orienting basis, tangent to $X$ at the point $x_{i}$; we add to the point $x_{i}$ the 
\enquote{$+$} sign, if the basis $\left( \tau_{(i)}^{X}, \tau_{(i)}^{S} \right)$ is an orienting basis for $M$ at the point $x_i$, and we add the \enquote{$-$} sign in the opposite case; denote this sign as $\text{sgn}\,x_{i}$. The index of intersection of submanifolds $X$ and $S$ is the integer number
$$
\displaystyle{ X\cdot S = \sum_{i} \text{sgn}\,x_{i} }.
$$

The index of intersection does not change if we replace the embeddings of $X$ and $S$ with one of the homotopic to them. Consider a small perturbation of the submanifold $S$ in its small neighborhood in the manifold $M$. As a result we obtain a submanifold $S'$, which intersects $S$ transversally. The algebraic number of points of intersection of the submanifolds $S$ and $S'$ is called an \emph{index of self-intersection} of the submanifold $S$. For the index of intersection we have a relation
\begin{equation}\label{eq:form_2_2}
\displaystyle{ \int_{S'} \eta_S = \int_S \eta_S = S\cdot S = \text{index of self-intersection of $S$.} }
\end{equation}

We need the following well-known proposition from symplectic topology (see \cite{makdaff}).
\begin{lem}\label{lem:index}
The index of self-intersection of a Lagrangian submanifold $X$ embedded to a symplectic 4-manifold $(M,\omega)$, where $X$ is oriented by the form $(-1)\omega \wedge \omega$, is equal to its Euler characteristic $\chi(X)$.
\end{lem}
\begin{proof}
There exists a canonical 1-form $\alpha$ on a cotangent bundle $T^{*}\!X$, the exterior differential of which defines a symplectic form on $T^{*}\!X$. In standard local coordinates $(\boldsymbol{p}, \boldsymbol{q})$, where $\boldsymbol{q} \in \mathbb{R}^2$ and $\boldsymbol{p} \in \mathbb{R}^2$ are the coordinates on the base and on the fiber respectively, the canonical 1-form and its differential are given by the formulas
$$
\displaystyle{ \alpha = p_1\,\text{d}\,q_1 + p_2\,\text{d}\,q_2,\quad \text{d}\,\alpha = \text{d}\,p_1 \wedge \text{d}\,q_1 + \text{d}\,p_2 \wedge \text{d}\,q_2 }
$$
According to Weinstein theorem (see, for example, \cite{makdaff}), there exists a symplectomorphism $\psi\colon U(X^2) \to V(X^2)$ of some neighborhood $U(X^2) \subset M$ of the submanifold $X$ in $M$ to a neighborhood $V(X) \subset T^{*}\!X$ of a zero-section of the foliation $T^{*}\!X$ with the symplectic form $\text{d}\,\alpha$, which transfers $X$ to the zero-section.

The form $(-1)\text{d}\,\alpha \wedge \text{d}\,\alpha$ defines on $T^{*}\!X$ a standard orientation, which is given by the order of the coordinates $(\boldsymbol{q},\boldsymbol{p})$.
For the standard orientation the index of self-intersection is equal to the sum of singular point's indexes of general vector field and is equal to $\chi(X)$.
\end{proof}

\begin{thm}\label{thm1}
Assume that an integrable Hamiltonian system at a 4-dimensional symplectic manifold $M$ has a focus-focus singularity of complexity $n \geq 2$. Then

a) $\pi_2(M) \neq 0$;

b) the dimension of De Rham cohomology group is $\text{\normalfont{dim}}\,H^2(M) \geq n-1$;

c) if the manifold $M$ is compact, then $\text{\normalfont{dim}}\,H^2(M) \geq n$.
\end{thm}

\begin{proof}
Let us orient the manifold $M$ using the form $(-1)\omega \wedge \omega$.
Then the self-intersection index of the Lagrangian submanifold $X \subset M$ is equal to its Euler characteristic $X \cdot X = \chi(X)$.

a) The singular fiber of the focus-focus singularity contains a Lagrangian sphere.
According to Lemma \ref{lem:index}, this Lagrangian sphere has a non-zero index of self-intersection and, hence, it is not homotopic to zero.

b) Consider a sequence of spheres $L_1,\ldots,L_n$ of the singular fiber of the focus type. Denote as $[\varphi_i]$ the classes, which are Poincare dual to the spheres $L_i$. Since the sphere is Lagrangian, its index of self-intersection is equal to +2. Then we obtain an equality
$$
\int_{L_{i}} \varphi_i  = 2.
$$
from the formula (\ref{eq:form_2_2}).

Thus, the classes $[\varphi_i]$ are non-trivial. Let us prove that they form a subspace of dimension no less than $n-1$ in $H^2(M)$. Exclude one sphere from the sequence. Let us orient spheres of a new sequence $L_1,\ldots,L_{n-1}$, which is not closed, in a such way that the index of intersection of the neighboring spheres becomes equal to +1. It is suffice to show that there exists no linear combination of classes $[\varphi_1],\ldots,[\varphi_{n-1}]$, such that the integral of it taken over any of the spheres from the sequence is equal to zero.
Using the integration over the spheres $L_i$, we obtain a system of equations for the coefficients
\begin{equation*}
\left\{
\begin{array}{l}
2 \lambda_1 + \lambda_2 = 0\\
\lambda_{i-1} + 2 \lambda_{i} + \lambda_{i+1} = 0,\ i = 2,\ldots,n-1\\
\lambda_{n-1} + 2\lambda_n = 0,
\end{array}
\right.
\end{equation*}
which has no nontrivial solution.

c) Consider a multiplication mapping on the symplectic form
$$
\wedge\,\omega: H^2(M) \rightarrow H^4(M) = \mathbb{R}
$$
According to the formula (\ref{eq:form_2_1}), all classes $\left[ \varphi_i \right]$ lie in the kernel of this mapping
$$
\int_M \omega \wedge \varphi_i = \int_{S_i} \omega = 0,
$$
and the image of the symplectic form is the whole space $H^4(M)$, because $\omega \wedge \omega$ is a volume form, which is not exact according to the compactness of $M$.
So, the classes $[\varphi_i]$, which are dual to the Lagrangian spheres of the singular fiber, do not form the cohomology class of the symplectic form and, hence, $\text{dim}\,H^2(M) \geq n$.
\end{proof}

Consider some examples of using this theorem. Note, that in the examples below the estimation of the complexity of the focus-focus singularity, obtained from the theorem \ref{thm1}, is not exact, but we can improve it by using other topological invariants.

\textbf{Examples:}
\begin{itemize}

\item $M = \mathbb{CP}^2$

Since $H^2(\mathbb{CP}^2) = \mathbb{R}$, then from (c) it results that $\mathbb{CP}^2$ admits only simple focus-focus singularities.

\item $M = M^{2}_{g_{1}} \times M^{2}_{g_{2}}$ (Cartesian product of spheres with handles)

If $g_1$ and $g_2$ are greater than zero, then $\pi_2(M^{2}_{g_{1}} \times M^{2}_{g_{2}}) = 0$. Then from (a) we obtain that the focus-focus singularities of $M$ are simple.

Assume that $g_2 = 0$ and $g_1 = g \neq 0$, that is $M^4 = M^{2}_{g} \times S^2$.
Since the group $H^2(M^{2}_{g} \times S^2) = \mathbb{R}^2$, then from (c) we have that in this case there are no singularities of complexity greater than 2.
Further, the group $\pi_2 (M^{2}_{g} \times S^2) = \mathbb{Z}$. 
Hence, the cohomology class of the sphere should have a zero-index of self-intersection. Thus, there are no complex focus-focus singularities in $M^{2}_{g} \times S^2$, while $g \neq 0$.

The most interesting case is the case, when $g_1 = g_2 = 0$, that is $M^4 = S^2 \times S^2$. Here $H^2(M^4) = \mathbb{R}^2$. From (c) we again obtain that there are no singularities of complexity of 3 or greater than 3. Can in this case exist singularities of complexity 2?
Let us orient the factors in some way, and define on $M^4$ the orientation given by the direct product. Let $a, b$ be the generators of group $H_2(S^2 \times S^2, \mathbb{Z}) = \mathbb{Z}^2$, implemented by the factors. For any cohomology class $\alpha a + \beta b$ we calculate the index of self-intersection
$$
(\alpha a + \beta b)\cdot(\alpha a + \beta b) = 2 \alpha \beta.
$$
We will assume here that the orientation, given by the form $(-1)\, \omega \wedge \omega$, is the opposite to the orientation of the direct product. In this case the index of self-intersection of the Lagrangian sphere is equal to ($-$2) and, hence, it is realized only by the class $a - b$. Then
\begin{equation}\label{eq:form_2_3}
\int_{a-b} \omega = \int_{a} \omega - \int_{b} \omega = 0.
\end{equation}

The integral of the symplectic form taken over the given class is equal to zero if and only if the areas of the spheres are the same. Condition (\ref{eq:form_2_3}) is the necessary condition of the existence of the focus-focus singularities of complexity 2 in $S^2 \times S^2$.
\item $M = M^{2}_{g} \times \mathbb{R}^2$ (Cartesian product of a sphere with handles and a plane)

When $g\neq 0$, we have $\pi_2(M^{2}_{g} \times \mathbb{R}^2) = 0$, and from (a) we obtain that
$M$ does not admit complicated focus-focus singularities.

When $g = 0$, the group $H^2(S^2 \times \mathbb{R}^2) = \mathbb{R}$. So, from (b) we have that there can not exist singularities of complexity more then 2 in $S^2 \times \mathbb{R}^2$.
Any sphere in $M$ is homotopic to some degree of factor-sphere and, hence, it has a zero-index of self-intersection. Finally we obtain that the manifold $S^2 \times \mathbb{R}^2$
admits only simple focus-focus singularities.
\end{itemize}

\begin{remark}
In contrast to the complicated singularities the simple focus-focus singularities are local, that is they can be realized in the space $(\mathbb{R}^4,\omega)$ with a standard symplectic form.
As an example we will mention the following pair of functions
$$
\displaystyle{H = \frac12 \left( p_{1}^2 + p_{2}^2 \right) + \left( q_{1}^2 + q_2^2 - 1 \right)^2,\ F= p_1 q_2 - p_2 q_1 }
$$
$$
\omega = \text{d}\,p_1 \wedge \text{d}\,q_1 + \text{d}\,p_2 \wedge \text{d}\,q_2
$$
The point $(0,0,0,0)$ is a singular point of the type focus-focus. The compactness of the focus singular fiber follows from the compactness of the energy levels $H=\text{const}$. Thus, there does not exist a topological obstruction to the existence of a simple focus-focus singularity.
\end{remark}

\section*{Applications to mechanics}\label{sec3}
In classical mechanical systems the phase is a cotangent bundle $T^{*}\! M^2$ to a 2-manifold $M^2$ (configuration space).
Assume that the configuration space is closed (compact and has no boundary), that is it is a sphere $M^2_g$ with $g$ handles in the orientable case, and it is a sphere $N^2_{\mu}$ with $\mu$ Mebius strips in the non-orientable case. Consider the following symplectic form on the phase space:
\begin{equation}\label{eq:sympl_form_eq}
\omega = \text{d}\,\alpha + \pi^{*} \varkappa.
\end{equation}
Here we denoted the standard symplectic form on the cotangent bundle as $\text{d}\,\alpha$, $\pi:T^{*}\!M^2 \rightarrow M^2$ is a projection mapping, and  $\varkappa$ is a closed 2-form on $M^2$ (the gyroscopic forces's form or magnetic term; see \cite{kharlamov1} for details).
\begin{thm}\label{thm2}
Consider an integrable Hamiltonian system on the cotangent bundle $T^{*}\!M^2$ to a 2-dimensional closed manifold $M^2$. The symplectic form $\omega$ on $T^{*}\!M^2$ has a form $\omega = \text{d}\,\alpha + \pi^{*}\varkappa.$ Then we have some cases:

1) If $M^2 = M^2_g$ and $g > 0$, or $M^2 = N^2_{\mu}$, then the system admits only simple focus-focus singularities.

2) If $g = 0$, then the system admits the singularities of complexity 1 and 2. The singularities of complexity 2 can appear in the system if and only if the form $\varkappa$ is exact.
\end{thm}
\begin{proof}

1) The phase space $T^{*}\!M^2$ contracts to the configuration space $M^2$.
The groups $\pi_2(M^2_g), \pi_2(N^2_{\mu})$ are non-trivial only in the case when $g = 0$ (sphere) and $\mu = 1$ (projective plane). In the case $M^2 = \mathbb{RP}^2$ we have $H^2(\mathbb{RP}^2) = 0$.
Now the statement follows from the points (a) and (b) of the theorem \ref{thm1}.

2) Since $H^2(S^2) = \mathbb{R}$, then from the point (b) of the theorem \ref{thm1} we obtain that the system does not admit singularities of complexity more than 2. Denote as $[\varphi]$ the dual class to one of the spheres $S$ of the singular fiber of a complex focus-focus singularity and as $[\omega]$ the cohomology class of the symplectic form. Then
\begin{equation}\label{eq:form_3_1}
[\omega] = k [\varphi]
\end{equation}
From the formula (\ref{eq:form_2_2}) and lemma \ref{lem:index} we have:
$
\displaystyle{\int_S \varphi = S\cdot S = 2}.
$
Since the sphere $S$ is Lagrangian,
$
\displaystyle{\int_S \omega = 0}.
$
Now from the condition (\ref{eq:form_3_1}) it follows that
$$
\displaystyle{\int_S \omega = k \int_S \varphi = 2 k = 0}.
$$
Thus, the coefficient $k$ is equal to zero, and the symplectic form is exact.
\end{proof}

Other series of integrable systems, interesting from the point of view of physics, is provided to us by the systems on the spaces dual to Lie algebras. Let $\mathfrak{g}$ be a finite-dimensional Lie algebra, and $\mathfrak{g}^{*}$ is a space of linear functions on $\mathfrak{g}$. Consider a basis
$e_1,\ldots,e_n$ of the algebra $\mathfrak{g}$. Let $c^{k}_{ij}$ be the structure constants of the algebra $\mathfrak{g}$ in this basis:
$$
\displaystyle{ \left[ e_i, e_j \right] = \sum_{k} c^{k}_{ij} e_k. }
$$
Consider linear coordinates $x_1,\ldots,x_n$ on $\mathfrak{g}^{*}$, corresponding to the basis $e_1,\ldots,e_n$. 
A natural Poisson bracket is defined on $\mathfrak{g}^{*}$ by the formula
$$
\displaystyle{ \left\{ f, g \right\} = \sum_{i,j,k} c^{k}_{ij} x_k \frac{\partial f}{\partial x_i} \frac{\partial g}{\partial x_j},\quad f,g \in C^{\infty}(\mathfrak{g}^{*}).}
$$
It is called a \emph{Lie-Poisson bracket} for the algebra $\mathfrak{g}$.

For an arbitrary smooth function $H$ on $\mathfrak{g}^{*}$ we can consider the equations
\begin{equation}\label{eq:coalgebra_eq}
\displaystyle{ \dot{x}_{i} = \left\{x_i, H\right\} },
\end{equation}
which are called the \emph{Euler equations} on $\mathfrak{g}^{*}$ with Hamiltonian function $H$.

The fibration is defined on the space $\mathfrak{g}^{*}$, the fibers of which are the orbits of the coadjoint representation of a Lie group. The restriction of the Lie-Poisson bracket on the given orbits is nondegenerate and, hence, it defines a natural symplectic form on these orbits. See \cite{olver} for details.

In the case when the orbits have a dimension 4, the integrability of the system (\ref{eq:coalgebra_eq}) means the existence of the integral, which does not depend on $H$ on the orbits. Let us restrict the hamiltonian and the additional integral on a certain base. Then we obtain an ordinary Hamiltonian system with two degrees of freedom and we get the opportunity to investigate the question of the existence of complicated focus-focus singularities in this system.

Many dynamical systems, describing some mechanical and physical processes, can be written as Hamiltonian systems on the dual spaces to Lie algebras. For example, Kirchhoff equations can be represented as a Hamiltonian system on $e(3)^{*}$. The problem of the motion by inertia of 4-dimensional rigid body can be written using the equations on $so(4)^{*}$. Let us consider two these cases in detail.

\subsection*{The case $\mathfrak{g} = e(3)$}
The Lie-Poisson bracket on $e(3)^{*}$, written in the coordinates $m_1, m_2, m_3, q_1, q_2, q_3$, has a form
\begin{equation}\label{eq:lie_bracket_e3}
\displaystyle{ \left\{ m_i, m_j \right\} = \sum_{k} \varepsilon_{ijk} m_k,\ \left\{ m_i, q_j \right\} = \sum_{k} \varepsilon_{ijk} q_k,\ \left\{ q_i, q_j \right\} = 0,}
\end{equation}
and the orbits of the coadjoint representaion are given by the equations
$$
\displaystyle{ \mathcal{O}_{q,m} = \left\{ f_1 = q_{1}^{2} + q_{2}^{2} + q_{3}^{2} = q^2,\ f_2 = m_1 q_1 + m_2 q_2 + m_3 q_3 = m q \right\}.}
$$
\begin{thm}\label{thm:thm_3_1}
In the integrable Hamiltonian systems on the orbits of the coadjoint representation on $e(3)^{*}$ can appear only simple focus-focus singularities, except, perhaps, the orbits with $m=0$.
On such orbits the focus-focus singularities have a complexity not greater than 2.
\end{thm}
\begin{proof}
Every regular ($q \neq 0$) orbit $\mathcal{O}_{q,m}$ is diffeomorphic to a cotangent bundle of the sphere $T^{*}\!S^2$. By changing the values of $m$ and $q$, we obtain different symplectic forms $\omega_{q,m}$ of the form (\ref{eq:sympl_form_eq}) on $T^{*}\!S^2$ and, hence, we can apply the theorem \ref{thm2}.
The cohomology class of the 2-form $\omega_{q,m}$, defining the symplectic structure,
has the form (see details in \cite{novikov})
$$
\displaystyle{\int_{S^2} \omega^{2}_{q,m} = 4 \pi m}.
$$
and vanishes only on orbits $m = 0$.

According to the paragraph 2 of the theorem \ref{thm2}, the manifold $T^{*}\!S^2$ does not admit focus-focus singularities of complexity more than 2, and the singularities of complexity 2 can appear only in the case when the symplectic form is exact, that is only if $m=0$.
\end{proof}

\subsection*{The case $\mathfrak{g} = so(4)$}
Consider coordinates $m_1, m_2, m_3, q_1, q_2, q_3$ on $so(4)^{*}$. The Lie-Poisson brackets, forming the basis, are defined as follows:
$$
\displaystyle{ \left\{ m_i, m_j \right\} = \sum_{k} \varepsilon_{ijk} m_k,\  \left\{ m_i, q_j \right\} = \sum_{k} \varepsilon_{ijk} q_k,\ \left\{ q_i, q_j \right\} = \sum_{k} \varepsilon_{ijk} m_k}.
$$
A pair of functions
\begin{equation}\label{eq:orbit_s04_1}
\displaystyle{ \mathcal{O}_{c, m} = \left\{
f_1 = q_{1}^{2} + q_{2}^{2} + q_{3}^{2} + m_{1}^{2} + m_{2}^{2} + m_{3}^{2} = c^2,\ f_2 = m_1 q_1 + m_2 q_2 + m_3 q_3 = m c
\right\} }
\end{equation}
sets a fibration on $so(4)^{*}$ with the fibers, which are the 4-dimensional orbits of the coadjoint representation.
\begin{thm}\label{thm:thm_3_2}
In integrable Hamiltonian systems on the orbits of the coadjoint representation on $so(4)^{*}$ can appear only simple focus-focus singularities, except, perhaps, the orbits with $m=0$.
On these orbits the focus-focus singularities have a complexity not greater than 2.
\end{thm}
\begin{proof}
The proof is based on using the well-known isomorphism of algebras $so(4)$ and $so(3) \oplus so(3)$. In the coordinates $s_1, s_2, s_3, p_1, p_2, p_3$ on $so(3)^{*} \oplus so(3)^{*}$ the Poisson brackets have a form
$$
\displaystyle{ \left\{ s_i, s_j \right\} = \sum_{k} \varepsilon_{ijk} s_k,\ \left\{ s_i, p_j \right\} = 0,\ \left\{ p_i, p_j \right\} = \sum_{k} \varepsilon_{ijk} p_k, }
$$
and the isomorphism is defined by the formulas
$$
\displaystyle{ m_i = s_i + p_i,\ q_i = p_i - s_i.}
$$
The equations for the orbits of the coadjoint representation in coordinates $s_i, p_i$ have a simple form:
\begin{equation}\label{eq:orbit_s04_2}
s_{1}^{2} + s_{2}^{2} + s_{3}^{2} = s^2,\quad p_{1}^{2} + p_{2}^{2} + p_{3}^{2} = p^2.
\end{equation}

From these formulas we can conclude that the orbits $\mathcal{O}_{s, p}$ in a general position are diffeomorphic to $S^2 \times S^2$.
Each of the equations (\ref{eq:orbit_s04_2}) defines a sphere in $so(3)^{*}$. These spheres are the orbits of the coadjoint representation of $so(3)$, and they have a natural symplectic form.
Assume that $\omega_s$ and $\omega_q$ are symplectic forms on the spheres $S^{2}_{s}$ and $S^{2}_{p}$ respectively. Then the symplectic form $\omega_{s, p}$ on $\mathcal{O}_{s, p} = S^{2}_{s} \times S^{2}_{p}$ has a form
$$
\omega_{s, p} = \pi_{s}^{*} \omega_s + \pi_{p}^{*} \omega_p,
$$
where the mappings $\pi_s$ and $\pi_p$ are projections on the factors.

Let us orient the spheres $S^{2}_{s}$ and $S^{2}_{p}$ by the forms $\omega_s$ and $\omega_p$ respectively. Then the orientation of the Cartesian product on $S^{2}_{s} \times S^{2}_{p}$ is set by the form $\pi_{s}^{*} \omega_s \wedge \pi_{p}^{*} \omega_p$. Then,
$$
\omega_{s, p} \wedge \omega_{s, p} = 2 \pi_{s}^{*} \omega_s \wedge \pi_{p}^{*} \omega_p,
$$
and, as noted above, the necessary condition for the existence of complicated focus-focus singularity is the coincidence of the symplectic areas of the spheres $S^{2}_{s}$ and $S^{2}_{p}$.

By direct calculations we can show that
$$
\displaystyle{ \int_{S^{2}_{s}} \omega_s = 4 \pi s,\ \displaystyle{ \int_{S^{2}_{p}}} \omega_p = 4 \pi p.}
$$
Thus, the complicated singularities can appear only on the orbits $s=p$. In terms of the values of given functions $f_1$ and $f_2$ these orbits are given by the condition $m = 0$.
\end{proof}

\section*{Model example of a mechanical system with a focus-focus singularity of complexity 2}
In this paragraph we will show that the estimates of complexity of a focus-focus singularity, given by the theorems \ref{thm2}, \ref{thm:thm_3_1} and \ref{thm:thm_3_2}, are exact, that is we will provide the examples of integrable systems with focus-focus singularities of complexity 2.
Consider on the space $e(3)^{*}$ a pair of functions
\begin{equation}\label{eq:form_4_1}
\displaystyle{H = \frac{1}{2} (m_1^2 + m_2^2 + m_3^2) + q_3^2,\ G = m_3,}
\end{equation}
which are in involution with respect to the bracket (\ref{eq:lie_bracket_e3}) and, hence, which define the integrable Hamiltonian systems with two degrees of freedom on the orbits $\mathcal{O}_{q,m}$. Denote as $h$ and $g$ the constant values of the integrals $H$ and $G$ respectively.

\begin{prop}
The system \eqref{eq:form_4_1} on the orbit $\mathcal{O}_{q, 0} = \left\{ f_1 = q^2, f_2 = 0 \right\}$ in $e(3)^{*}$ has a focus-focus singularity of complexity 2. The singular focus fiber $L$ is given by the equations $L = \left\{ H = q^2, G = 0 \right\}$.
\end{prop}
\begin{proof}
The points $\left\{q_1 = q_2 = m_1 = m_2 = 0, m_3 = \varepsilon m, q_3 = \varepsilon q \right\},\ \varepsilon = \pm 1$ when $\abs{m} < 2 \sqrt{2} q$ are singular points of focus-focus type.
When $m = 0$, these points lie on the same level $L = \left\{ h = q^2, g = 0 \right\}$.

Let us prove that the singular fiber $L$, which contains the points $P_{+}$ и $P_{-}$, is connected, that is the orbit $\mathcal{O}_{q,0}$ actually contains a complicated focus-focus singularity, and not a pair of simple singularities. The fiber $L$ in the space $\mathbb{R}^6(m_1, m_2, m_3, q_1, q_2, q_3)$ is defined by the system of four equations
$$
L: \left\{
\begin{aligned}
f_1 &=  q_1^2 + q_2^2 + q_3^2 = q^2 \\
f_2 &= m_1 q_1 + m_2 q_2 + m_3 q_3 = 0 \\
H &= \dfrac{1}{2} (m_1^2 + m_2^2 + m_3^2) + q_3^2 = q^2 \\
G &= m_3 = 0
\end{aligned}
\right.,
$$
which has an explicit solution, namely:
vector $\boldsymbol{q} = (q_1, q_2, q_3)$ is any vector on the sphere
$$
\displaystyle{ q_1^2 + q_2^2 + q_3^2 = q^2},
$$
and for any such $\boldsymbol{q}$ we have
$$
\displaystyle{m_1 = \pm \sqrt{2} q_2,\quad m_2 = \mp \sqrt{2} q_1,\quad m_3 = 0}.
$$
It is clear that one fixed sign (both upper or both lower) gives a sphere, and these spheres intersect in two points $P_{+}$ and $P_{-}$.

Thus, the system \eqref{eq:form_4_1} is an example of an integrable system on $\mathcal{O}_{q, 0}$ with a focus-focus singularity of complexity 2.
Moreover, when we pass to the neighboring orbits $\mathcal{O}_{q, m}$, that means in case of a small integrable perturbation, the complicated focus-focus singularity splits into two simple singularities.
\end{proof}

The functions (\ref{eq:form_4_1}) are in involution with respect not only to the bracket on $e(3)^{*}$,
but also to the whole family of the Poisson brackets of the form
\begin{equation}\label{eq:form_4_2}
\displaystyle{ \left\{ m_i, m_j \right\} = \sum_{k} \varepsilon_{ijk} m_k,\ \left\{ m_i, q_j \right\} = \sum_{k} \varepsilon_{ijk} q_k,\ \left\{ q_i, q_j \right\} = \lambda \sum_{k} \varepsilon_{ijk} m_k}.
\end{equation}
The equations for the orbits of the coadjoint representation for this bracket depend on the parameter $\lambda$ of the family of the brackets as follows:
$$
\displaystyle{ \mathcal{O}_{p,m} = \left\{ f_1 = \lambda (m_1^2 + m_2^2 + m_3^2) + (q_1^2 + q_2^2 + q_3^2) = c^2,\quad f_2 = m_1 q_1 + m_2 q_2 + m_3 q_3 = m c \right\} }
$$
When $\lambda > 0$, we can make the following change of variables: $m_i = \widetilde{m}_i, q_i = \sqrt{\lambda}\, \widetilde{q}_i$, and after this we can figure out that the bracket (\ref{eq:form_4_2}) corresponds to the algebra $so(4)$.
\begin{prop}
The system \eqref{eq:form_4_1} on the orbit $\mathcal{O}_{c, 0} = \left\{ f_1 = c^2, f_2 = 0 \right\}$ of the bracket \eqref{eq:form_4_2} has a focus-focus singularity of complexity 2. The singular focus fiber $L$ is given by the equations $L = \left\{ H = c^2, G = 0 \right\}$.
\end{prop}
\begin{proof}
Consider the integrable system (\ref{eq:form_4_1}) on the regular ($c \neq 0$) orbit $\mathcal{O}_{q,0}$. The points
$$
P_{\pm} = \left\{q_1 = q_2 = m_1 = m_2 = m_3 = 0, q_3 = \varepsilon c \right\},\ \varepsilon = \pm 1$$
are points of focus-focus type for sufficiently small $\lambda$.

As in the previous example the singular fiber is almost explicitly defined: vector $\boldsymbol{q} = (q_1, q_2, q_3)$ is any vector on the ellipsoid
$$
\displaystyle{ \frac{1}{1 - 2 \lambda} (q_1^2 + q_2^2) + q_3^2 = c^2},
$$
and for any such $\boldsymbol{q}$ we have
$$
\displaystyle{m_1 = \pm q_2 \sqrt{\dfrac{2}{1 - 2 \lambda}},\quad m_2 = \mp q_1 \sqrt{\dfrac{2}{1 - 2 \lambda}},\quad m_3 = 0}.
$$

As previously we have got two spheres which are intersecting transversally at the points $P_{+}$ and $P_{-}$.
So, the system \eqref{eq:form_4_1} is integrable Hamiltonian system on the orbit $\mathcal{O}_{c,0}$ which has got a focus-focus singularity of complexity 2.
\end{proof}
\section*{Acknowledgements}
The author would like to thank Prof. A.A.\,Oshemkov for the helpful discussion.


\end{document}